\documentclass[12pt]{article}
\usepackage{graphicx}
\usepackage[ruled,vlined]{algorithm2e}
\usepackage{appendix}
\usepackage{amsmath}
\usepackage{amssymb}
\usepackage{amsthm}
\usepackage{multirow}
\usepackage{color}
\usepackage{longtable}
\usepackage{array}
\usepackage{url}
\usepackage{booktabs}
\usepackage{float}
\usepackage{mathtools}
\usepackage{tikz}
\usepackage{caption}

\newtheorem{theorem}{Theorem}[section]
\newtheorem{definition}[theorem]{Definition}
\newtheorem{lemma}[theorem]{Lemma}

\newtheorem{corollary}[theorem]{Corollary}
\newtheorem{proposition}[theorem]{Proposition}

\newtheorem{conjecture}[theorem]{Conjecture}

\theoremstyle{definition}

\title{On linear diameter perfect Lee codes with diameter 6}
\author{Tao Zhang$^{\text{a,}}$\thanks{Email address:  zhant220@163.com.}~ and  Gennian Ge$^{\text{b,}}$\thanks{Email address:  gnge@zju.edu.cn.}\\
\footnotesize $^{\text{a}}$ Zhejiang Lab, Hangzhou 311100, China. \\
\footnotesize $^{\text{b}}$ School of Mathematics Sciences, Capital Normal University, Beijing 100048, China.\\}

\begin{document}

\date{}

\maketitle

\begin{abstract}
In 1968, Golomb and Welch conjectured that there is no perfect Lee codes with radius $r\ge2$ and dimension $n\ge3$. A diameter perfect
code is a  natural generalization of the perfect code. In 2011, Etzion (IEEE Trans. Inform. Theory, 57(11): 7473--7481, 2011) proposed the following problem: Are there diameter perfect Lee (DPL, for short) codes with diameter greater than four besides the $DPL(3,6)$ code? Later, Horak and AlBdaiwi (IEEE Trans. Inform. Theory, 58(8): 5490--5499, 2012) conjectured that there are no $DPL(n,d)$ codes for dimension $n\ge3$ and diameter $d>4$ except for $(n,d)=(3,6)$. In this paper, we give a counterexample to this conjecture. Moreover, we prove that for $n\ge3$, there is a linear $DPL(n,6)$ code if and only if $n=3,11$.

\medskip


\smallskip

\end{abstract}

\section{Introduction}
The Lee metric was first introduced in \cite{L58,U57} for transmission of signals. For two vectors $u=(u_1,\dots,u_n), v=(v_1,\dots,v_n)\in\mathbb{Z}^{n}$, their \emph{Lee distance} is defined by
\[d_{L}(u,v)=\sum_{i=1}^{n}|u_i-v_i|.\]
A \emph{Lee code} $C$ in $\mathbb{Z}^{n}$ is a subset of $\mathbb{Z}^{n}$ endowed with Lee distance. If $C$ is also a subgroup of $\mathbb{Z}^{n}$, then $C$ is called a \emph{linear Lee code}. Lee codes have been well studied due to many practical applications, for example, constrained and partial-response channel \cite{RS1994}, flash memory \cite{JSB2010,S2012}, interleaving schemes \cite{BBV1998} and multidimensional burst error correction \cite{EY2009}.

Let $S_{n,r}(u)$ denote the \emph{Lee sphere} of radius $r$ centered at $u$, that is,
\[S_{n,r}(u)=\{v\in\mathbb{Z}^{n}:\ d_{L}(u,v)\le r\}.\]
The \emph{packing radius} $r(C)$ of a Lee code $C$ is the greatest integer $r$ such that $S_{n,r}(u)\cap S_{n,r}(v)=\emptyset$ holds for all $u,v\in C$. The \emph{covering radius} $R(C)$ of a Lee code $C$ is the smallest integer $r$ such that $\cup_{u\in C}S_{n,r}(u)=\mathbb{Z}^{n}$. A Lee code $C\subseteq\mathbb{Z}^{n}$ is called \emph{perfect} if $r(C)=R(C)=r$, and then $C$ is called a $PL(n,r)$ code.

Another way to define perfect codes is by tiling.
For a subset $V\subset\mathbb{Z}^{n}$, if there exists $W\subset\mathbb{Z}^{n}$ such that $\mathbb{Z}^{n}=V\bigoplus W$, then we say that $\mathbb{Z}^{n}$ can be \emph{tiled} by $V$. It is easy to see that there is a $PL(n,r)$ code if and only if $\mathbb{Z}^{n}$ can be tiled by $S_{n,r}(0)$. The well-known \emph{Golomb-Welch conjecture} states that
\begin{conjecture}\rm{\cite{GW70}}
For $n\ge3$ and $r\ge2$, there is no $PL(n,r)$ code.
\end{conjecture}
In \cite{GW70}, Golomb and Welch showed that there is no $PL(n,r)$ code for $n\ge3$ and $e\ge e_{n}$, where $e_n$ is not specified. Post \cite{P75} proved the nonexistence of $PL(n,r)$ codes for $n\ge6$ and $r\ge\frac{\sqrt{2}}{2}n-\frac{1}{4}(3\sqrt{2}-2)$. Later, Lepist{\"o} \cite{L81} showed that $PL(n,r)$ codes must satisfy $n\ge\frac{(r+2)^2}{2.1}$, where $r\ge285$.

For some small dimensions, Golomb-Welch conjecture has been solved. Gravier et al. \cite{GMP98} settled the dimension 3 case. Dimension 4 was resolved in \cite{S07}, and dimension 5 was proved in \cite{H09E}. Horak \cite{H09} showed the nonexistence of perfect Lee codes for $n=6$ and $r=2$.

Researchers also considered the conjecture with small radius. Under certain conditions, the nonexistence of $PL(n,r)$ (or linear $PL(n,r)$) codes has been proved in \cite{K17,Q20,ZG17,ZZ19} for $r=2,3,4$. In \cite{LZ20}, Leung and Zhou completely solved the existence of linear $PL(n,2)$ codes. In \cite{XZ2022,ZZ2022}, the authors showed that almost perfect linear Lee codes of radius 2 only exist for small dimensions. For more information about Golomb-Welch conjecture, we refer the readers to
a recent survey \cite{HK18}.

A natural generalization of perfect code includes the  diameter perfect code. Let $(V,d)$ be a metric space. Then a set $C\subseteq V$ is a \emph{diameter $d$ code} if $d(u,v)\ge d$ for any $u,v\in C$, and a set $A\subseteq V$ is a \emph{diameter $d$ anticode} if $d(u,v)\le d$ for all $u,v\in A$. In \cite{AAK01}, the authors proved that if $V$ is a distance regular graph, $C$ is a diameter $d$ code and $A$ is a diameter $d-1$ anticode, then $|C||A|\le |V|$. A diameter $d$ code that attains above bound with equality is called a diameter $d$ perfect code. For more about perfect codes (and diameter perfect codes), we refer the readers to \cite{E22}.

However, the above definition cannot be extended to infinite space. In order to define diameter perfect Lee codes, we need the following notations. Let $\mathcal{S}=\{S_i: i\in I\}$ be a family of subsets of $\mathbb{Z}^{n}$. Then a set $T\subset\mathbb{Z}^{n}$ is called a transversal of $\mathcal{S}$ if $|T\cap S_i|=1$ for each $i\in I$, and $T\cap S_i\ne T\cap S_j$ for all $i\ne j$.

\begin{definition}
  Let $C\subseteq\mathbb{Z}^{n}$. Then $C$ is a diameter $d$ perfect Lee code if $C$ is a diameter $d$ code, and there is a tiling $\mathcal{T}=\{W_i: i\in I\}$ of $\mathbb{Z}^{n}$ by the anticode of diameter $d-1$ of maximum size such that $C$ is a transversal of $\mathcal{T}$. The diameter $d$ perfect Lee code in $\mathbb{Z}^{n}$ will be denoted by $DPL(n,d)$.
\end{definition}
For $d$ odd, the anticode of diameter $d-1$ of the maximum size is the Lee sphere $S_{n,\frac{d-1}{2}}$. Hence $DPL(n,d)$ is the same as $PL(n,\frac{d-1}{2})$ for $d$ odd.

For $u,v\in\mathbb{Z}^{n}$ with $d_L(u,v)=1$, the double Lee sphere is defined by
\[DS_{n,r}(u,v)=S_{n,r}(u)\cup S_{n,r}(v).\]
The size of double Lee sphere is \cite{Etzion2003}
\[|DS_{n,r}|=\sum_{i=0}^{\min\{n-1,r\}}2^{i+1}\binom{n-1}{i}\binom{r+1}{i+1}.\]
It was pointed out in \cite{AB08} that, for $d$ even, the anticode of diameter $d-1$ of maximum size is the double Lee sphere $DS_{n,r}$ with $r=\frac{d-2}{2}$.

The following conjecture is an extension of Golomb-Welch conjecture, which was first proposed by Etzion \cite{E11} and formally presented in \cite{HA12}.
\begin{conjecture}
  There is no $DPL(n,d)$ code for $n\ge3$ and $d>4$ with the exception of the pair $(n,d)=(3,6)$.
\end{conjecture}

In this paper, we give a counterexample to this conjecture. Moreover, we prove the following result.
\begin{theorem}\label{mainthm}
For $n\ge3$, there is a linear $DPL(n,6)$ code if and only if $n=3,11$.
\end{theorem}

This paper is organized as follows. In Section~\ref{prelimi}, we first give some basics about group rings and partial difference sets, then we provide a group ring representation of linear $DPL(n,6)$ codes. In Section~\ref{mainsec}, we prove our main result.
\section{Preliminaries}\label{prelimi}
\subsection{Group ring and partial difference set}
Let $G$ be a finite abelian group (written multiplicatively). The group ring $\mathbb{Z}[G]$ is a free abelian group with a basis $\{g:\ g\in G\}$. For any $A\in\mathbb{Z}[G]$, $A$ can be written as formal sum $A=\sum_{g\in G}a_gg$, where $a_g\in\mathbb{Z}$. For any $A=\sum_{g\in G}a_gg\in\mathbb{Z}[G]$ and $t\in\mathbb{Z}$, we define
\[A^{(t)}=\sum_{g\in G}a_gg^t.\]
Addition, subtraction and multiplication in group rings are defined as:
\[\sum_{g\in G}a_gg\pm\sum_{g\in G}b_gg=\sum_{g\in G}(a_g\pm b_g)g,\]
and
\[\sum_{g\in G}a_gg\sum_{g\in G}b_gg=\sum_{g\in G}(\sum_{h\in G}a_hb_{h^{-1}g})g.\]
For any set $A$ whose elements belong to $G$ ($A$ may be a multiset), we can identify $A$ with the group ring element $\sum_{g\in G}a_gg$, where $a_g$ is the multiplicity of $g$ appearing in $A$.
For a finite abelian group $G$, its character group is denoted by $\widehat{G}$. For any $A=\sum_{g\in G}a_gg\in\mathbb{Z}[G]$ and $\chi\in\widehat{G}$, define
\[\chi(A)=\sum_{g\in G}a_g\chi(g).\]
Group rings are widely used in the study of difference sets and related topics, see \cite{P95,S02} and the references therein.

We also need some results about partial difference sets. We first give its definition.
\begin{definition}
A $k$-element subset $D$ of a finite group $G$ (written multiplicatively) of order $v$ is called a $(v,k,\lambda,\mu)$-partial difference set in $G$ if the multiset $\{d_1d_{2}^{-1}:\ d_1,d_2\in D,d_1\ne d_2\}$ contains each nonidentity element of $D$ exactly $\lambda$ times and each nonidentity element of $G\backslash D$ exactly $\mu$ times. When $e\not\in D$ and $D^{(-1)}=D$, then $D$ is called a regular partial difference set.
\end{definition}
In the language of group ring, $D$ is a $(v,k,\lambda,\mu)$ partial difference set in $G$ if and only if
\[DD^{(-1)}=\mu G+(\lambda-\mu)D+\gamma e,\]
where $\gamma=k-\mu$ if $e\notin D$ and $\gamma=k-\lambda$ if $e\in D$.
See \cite{M1994} for the background on partial difference sets. The following result can be found in \cite{M1984}.
\begin{lemma}{\rm{\cite{M1984}}}\label{lemma4}
  If there exists a nontrivial regular $(v,k,\lambda,\mu)$ partial difference set in an abelian group, then $v,\Delta$ and $v^2/\Delta$ must have the same prime divisors, where $\Delta=(\lambda-\mu)^2+4(k-\mu)$.
\end{lemma}
\subsection{Group ring representation of linear $DPL(n,6)$ codes}
In this paper, we focus on the existence of linear $DPL(n,d)$ codes. The following theorem says that there is a linear $DPL(n,d)$ code if and only if there is a lattice tiling of $\mathbb{Z}^{n}$ by a fixed double Lee sphere $DS_{n,\frac{d-2}{2}}$.
\begin{theorem}{\rm{\cite{HA12}}}
Let $d$ be an even integer and $C$ be a linear $DPL(n,d)$ in $\mathbb{Z}^{n}$. Let $W$ be a double Lee sphere $DS_{n,r}$ with $r=\frac{d-2}{2}$. Then $\mathcal{T}=\{W+c: c\in C\}$ is a lattice tiling of $\mathbb{Z}^{n}$ by double Lee sphere $W$.
\end{theorem}
The following theorem can be found in \cite{HA12}, which establishes the connection between lattice tilings of $\mathbb{Z}^n$ and finite abelian groups.
\begin{theorem}{\rm{\cite{HA12}}}\label{tilinggroup}
  Let $V$ be a subset of $\mathbb{Z}^{n}$. Then there is a lattice tiling $\mathcal{T}$ of $\mathbb{Z}^n$ by $V$ if and only if there are both an abelian group $G$ of order $|V|$ and a homomorphism $\phi:\mathbb{Z}^n\rightarrow G$ so that the restriction of $\phi$ to $V$ is a bijection.
\end{theorem}
The following lemma can be found in \cite{SS09}.
\begin{lemma}{\rm{\cite[Theorem 3.19]{SS09}}}\label{directsum}
If a finite abelian group $G$ is a direct sum of subgroups $H,K$ of relatively prime orders and $G=A+B$ is a factorization, where $|A|=|H|$, $|B|=|K|$, then $G=H+B=A+K$.
\end{lemma}

Now we translate the existence of linear $DPL(n,6)$ codes into group ring equations.
\begin{theorem}\label{groupring}
  Let $n\ge3$, then there exists a linear $DPL(n,6)$ code if and only if there exists a finite abelian group $H$ of order $2n^2+1$ and $T\subseteq H$ satisfying
  \begin{enumerate}
    \item[(1)] $|T|=2n$,
    \item[(2)] $T=T^{(-1)}$,
    \item[(3)] $T^2=2H-T^{(2)}+(2n-2)e$.
  \end{enumerate}
\end{theorem}
\begin{proof}
Let $e_i$, $i=1,2,\dots,n$, be a fixed orthonormal basis of $\mathbb{Z}^{n}$.
  By Theorem~\ref{tilinggroup}, there exists a linear $DPL(n,6)$ code if and only if there are both an abelian group $G$ (written multiplicatively) of order $4n^2+2$ and a homomorphism $\phi:\mathbb{Z}^n\rightarrow G$ such that the restriction of $\phi$ to $DS_{n,2}(0,e_1)$ is a bijection. Since the homomorphism $\phi$ is determined by the values $\phi(e_i)$, $i=1,\dots,n$, then there exists a linear $DPL(n,6)$ code if and only if there exists an $n$-subset $\{a_1,a_2,\dots,a_n\}\subset G$ (let $\phi(e_i)=a_i$) such that
  \begin{align*}
  G=&\{e,a_{1}^{\pm1},a_{1}^{\pm2},a_{1}^{3}\}\cup\{a_{i}^{\pm1},a_{1}^{\pm1}a_{i}^{\pm1},a_{1}^{2}a_{i}^{\pm1},a_{i}^{\pm2},a_1a_{i}^{\pm2}:\ 2\le i\le n\}\\
  &\cup\{a_{i}^{\pm1}a_{j}^{\pm1},a_{1}a_{i}^{\pm1}a_{j}^{\pm1}:\ 2\le i<j\le n\}.
  \end{align*}
  In the language of group ring, the above equation can be written as
  \begin{align*}
  G=&e+a_1+a_{1}^{2}+a_{1}^{-1}+a_{1}^{-2}+a_{1}^{3}+(e+a_1+a_{1}^{-1}+a_{1}^{2})\sum_{i=2}^{n}(a_i+a_{i}^{-1})+(e+a_1)\sum_{i=2}^{n}(a_{i}^{2}+a_{i}^{-2})\\
  &+(e+a_1)\sum_{2\le i<j\le n}(a_i+a_{i}^{-1})(a_j+a_{j}^{-1})\\
  =&(e+a_1)(e+a_{1}^{2}+a_{1}^{-2}+(a_1+a_{1}^{-1})\sum_{i=2}^{n}(a_i+a_{i}^{-1})+\sum_{i=2}^{n}(a_{i}^{2}+a_{i}^{-2})+\sum_{2\le i<j\le n}(a_i+a_{i}^{-1})(a_j+a_{j}^{-1}))\\
  =&(e+a_1)(e+\sum_{i=1}^{n}(a_{i}^{2}+a_{i}^{-2})+\sum_{1\le i<j\le n}(a_i+a_{i}^{-1})(a_j+a_{j}^{-1})).
  \end{align*}
  Since $|G|=4n^2+2$ and $\gcd(2,2n^2+1)=1$, then $\mathbb{Z}_{2}$ is a subgroup of $G$. Let $H=G/\mathbb{Z}_{2}$ and $\varphi:G\rightarrow G/\mathbb{Z}_{2}$ be the natural homomorphism, then $|H|=2n^2+1$.
   By Lemma~\ref{directsum}, we have
  \[\varphi(e)+\sum_{i=1}^{n}(\varphi(a_{i})^{2}+\varphi(a_{i})^{-2})+\sum_{1\le i<j\le n}(\varphi(a_i)+\varphi(a_{i})^{-1})(\varphi(a_j)+\varphi(a_{j})^{-1})=H.\]
  Let $T=\sum_{i=1}^{n}(\varphi(a_{i})+\varphi(a_{i})^{-1})$, then we can compute to get that
  \begin{align*}
  T^2=&(\sum_{i=1}^{n}(\varphi(a_{i})+\varphi(a_{i})^{-1}))^2\\
  =&\sum_{i=1}^{n}(\varphi(a_{i})^2+\varphi(a_{i})^{-2})+2\sum_{1\le i<j\le n}(\varphi(a_i)+\varphi(a_{i})^{-1})(\varphi(a_j)+\varphi(a_{j})^{-1})+2n\varphi(e)\\
  =&2H-T^{(2)}+(2n-2)\varphi(e).\qedhere
  \end{align*}
\end{proof}

\begin{lemma}\label{lemma6}
Let $H$ be a finite abelian group of order $2n^2+1$ and $T\subset H$ satisfying conditions in Theorem~\ref{groupring}, then
for any nontrivial character $\chi\in\widehat{H}$, $\chi(T)\ne0$.
\end{lemma}
\begin{proof}
  If $\chi(T)=0$, by Theorem~\ref{groupring}, we have $\chi(T^{(2)})=2n-2$. This is impossible since $|T^{(2)}|=2n$ and $\gcd(2,|H|)=1$.
\end{proof}

\section{Proof of Theorem~\ref{mainthm}}\label{mainsec}
In this section, we will give an explicit construction of linear $DPL(11,6)$ code and prove that
 for any $n\ge4$, $n\ne11$, there does not exist any subset $T$ in a finite abelian group $H$ satisfying conditions in Theorem~\ref{groupring}.

 We will first study $T^{(2)}T$, by modulo 3, we can get some restrictions on the coefficients of elements of $T^{(2)}T$. This can solve the cases $n\equiv1\pmod{3}$ and $n\equiv2\pmod{3}$ except for $n=5,8,11$. However, we cannot get a contradiction when $n\equiv0\pmod{3}$. Hence we consider $TT^{(4)}\pmod{5}$ and $TT^{(5)}\pmod{5}$. Small dimensions will be solved in the last subsection.

Let $n\ge3$ be an integer. Suppose that $H$ is a finite abelian group with order $2n^2+1$, and $T\subset H$ satisfying conditions in Theorem~\ref{groupring}. We write $T$ as $T=\sum_{i=1}^{n}(a_i+a_{i}^{-1})$, then
\begin{align}\label{eq1}
  H=e+\sum_{i=1}^{n}(a_{i}^{2}+a_{i}^{-2})+\sum_{1\le i<j\le n}(a_i+a_{i}^{-1})(a_j+a_{j}^{-1}).
\end{align}

 We first prove some lemmas.
\begin{lemma}\label{lemma3}
  $T\cap T^{(3)}=\emptyset$.
\end{lemma}
\begin{proof}
 If there exist $a_i,a_j\in T$ with $j\ne i$ such that $a_i=a_{j}^{3}\in T\cap T^{(3)}$, then $a_{j}^{2}=a_{i}a_{j}^{-1}$. By Equation~(\ref{eq1}), it can be seen that $a_{j}^{2}$ appears twice in $H$, which is a contradiction.

 If there exists $a_{i}\in T$ such that $a_i=a_{i}^{3}\in T\cap T^{(3)}$, then $a_{i}^{2}=e$. Since $\gcd(2,|H|)=1$, then $a_i=e$, which contradicts to $e\notin T$.

 If there exists $a_{i}\in T$ such that $a_i=a_{i}^{-3}\in T\cap T^{(3)}$, then $a_{i}^{4}=e$. Similarly as above, we have $a_i=e$, which is also a contradiction.
\end{proof}
\begin{lemma}\label{lemma1}
  For any $t\ne e$, we have
 \[|\{(t_1,t_2)\in T\times T: t_1t_2=t\}|= \begin{cases}
    1, & \mbox{if } t\in T^{(2)}; \\
    2, & \mbox{if }t\notin T^{(2)}.
  \end{cases}\]
\end{lemma}
\begin{proof}
  The proof is directly from Condition (3) of Theorem~\ref{groupring}.
\end{proof}
Since $T^2=2H-T^{(2)}+(2n-2)e$, then
\begin{align}\label{eq4}
TT^{(2)}=4nH+(2n-2)T-T^{3}.
\end{align}
Writing $TT^{(2)}$ as
\begin{align}\label{eq2}
TT^{(2)}=\sum_{i=0}^{M}iX_i,
\end{align}
where $X_i$ $(i=0,1,\dots,M)$ form a partition of group $H$, i.e.,
\begin{align}\label{eq3}
H=\sum_{i=0}^{M}X_i.
\end{align}
Then we have the following lemma.
\begin{lemma}\label{lemma2}
  \begin{enumerate}
    \item [(1)] $\sum_{i=1}^{M}i|X_i|=4n^2$,
    \item [(2)] $\sum_{i=0}^{M}|X_i|=2n^2+1$,
    \item [(3)] $\sum_{i=1}^{M}|X_i|=4n-\beta+\sum_{i\ge3}\frac{(i-1)(i-2)}{2}|X_i|$, where $2\beta=|T\cap T^{(2)}|$.
  \end{enumerate}
\end{lemma}
\begin{proof}
  The first two equations are directly from Equations~(\ref{eq2}) and (\ref{eq3}). Now we count the size of distinct elements in $T^{(2)}T$.
Let $T^{(2)}=\sum_{i=1}^{2n}b_i$ with $b_{i}^{-1}=b_{n+i}$ for $i=1,2,\dots,n$.
  By Lemma~\ref{lemma1}, we have
   \[|b_iT\cap b_jT|= \begin{cases}
    1, & \mbox{if } b_ib_{j}^{-1}\in T^{(2)}; \\
    2, & \mbox{if }b_ib_{j}^{-1}\notin T^{(2)}.
  \end{cases}\]
  Let $|T\cap T^{(2)}|=2\beta$ and $\alpha=|\{(i,j): 1\le i<j\le 2n, b_ib_{j}^{-1}\in T^{(2)}\}|$. Let $N(t)=|\{(i,j): 1\le i<j\le 2n, b_ib_{j}^{-1}=t\}|$, then $\alpha=\sum_{t\in T^{(2)}}N(t)$. Since $\gcd(2n^2+1,2)=1$, then $T^{(2)}$ also satisfies conditions in Theorem~\ref{groupring}. By the definition of $b_i$ and Lemma~\ref{lemma1}, we have
     \[N(t)+N(t^{-1})= \begin{cases}
    1, & \mbox{if } t\in T^{(2)}\cap T^{(4)}; \\
    2, & \mbox{if } t\in T^{(2)}\backslash T^{(4)}.
  \end{cases}\]
  Hence, $\alpha=|T^{(2)}\backslash T^{(4)}|+\frac{|T^{(2)}\cap T^{(4)}|}{2}=|T\backslash T^{(2)}|+\frac{|T\cap T^{(2)}|}{2}=2n-\beta$.

By the inclusion-exclusion principle, we have
  \begin{align*}
  \sum_{i=1}^{M}|X_i|=&|T^{(2)}T|\\
  =&\sum_{i=1}^{2n}|b_iT|-\sum_{1\le i<j\le 2n}|b_iT\cap b_jT|+\sum_{r\ge3}\sum_{1\le i_1<\dots<i_r\le 2n}(-1)^{r-1}|b_{i_1}T\cap\dots\cap b_{i_r}T|\\
  =&4n^2-\alpha-2(\binom{2n}{2}-\alpha)+\sum_{r\ge3}\sum_{1\le i_1<\dots<i_r\le 2n}(-1)^{r-1}|b_{i_1}T\cap\dots\cap b_{i_r}T|.
  \end{align*}
  Suppose that $g\in b_{i_1}T\cap\dots\cap b_{i_r}T$ for some $r\ge 3$, then $g\in X_i$ for some $i\ge3$. The contribution for $g$ in the sum $\sum_{r\ge3}\sum_{1\le i_1<\dots<i_r\le 2n}(-1)^{r-1}|b_{i_1}T\cap\dots\cap b_{i_r}T|$ is
  \[(-1)^{3-1}\binom{i}{3}+(-1)^{4-1}\binom{i}{4}+\cdots=\binom{i}{2}-\binom{i}{1}+\binom{i}{0}=\frac{(i-1)(i-2)}{2}.\]
  Hence, we have
  \begin{align*}
  \sum_{i=1}^{M}|X_i|=&4n^2-\alpha-2(\binom{2n}{2}-\alpha)+\sum_{r\ge3}\sum_{1\le i_1<\dots<i_r\le 2n}(-1)^{r-1}|b_{i_1}T\cap\dots\cap b_{i_r}T|\\
  =&4n^2-\alpha-2(\binom{2n}{2}-\alpha)+\sum_{i\ge3}|X_i|\frac{(i-1)(i-2)}{2}\\
  =&4n-\beta+\sum_{i\ge3}\frac{(i-1)(i-2)}{2}|X_i|.\qedhere
  \end{align*}
\end{proof}
Now we divide our discussion into three cases.

\subsection{$n\equiv1\pmod{3}$}
\begin{proposition}\label{prop1}
  Theorem~\ref{mainthm} holds for $n\equiv1\pmod{3}$.
\end{proposition}
\begin{proof}
By Equation (\ref{eq4}), we have
\[TT^{(2)}\equiv H+2T^{(3)}\pmod{3}.\]
Then we get
\begin{align*}
&\sum_{i\ge0}|X_{3i}|=a_1,\\
&\sum_{i\ge0}|X_{3i+2}|=a_2,\\
&\sum_{i\ge0}|X_{3i+1}|=2n^2+1-a_1-a_2,
\end{align*}
where
\begin{align*}
&a_1=|\{h\in T^{(3)}:\ h\text{ appears }3i+1 \text{ times in }T^{(3)}\text{ for some }i\}|,\\
&a_2=|\{h\in T^{(3)}:\ h\text{ appears }3i+2 \text{ times in }T^{(3)}\text{ for some }i\}|.
\end{align*}
Hence $a_1+a_2\le 2n$. Since for any $t\in T$, $t^3=t\cdot t^2$, then any element in $T^{(3)}$ appears at least once in $TT^{(2)}$. Note that $\cup_{i\ge0}X_{3i}\subseteq T^{(3)}$, then $|X_0|=0$. By Lemma~\ref{lemma2}, we obtain
\begin{align*}
4n^2+4n-\beta=&\sum_{i=1}^{M}i|X_i|+\sum_{i=1}^{M}|X_i|-\sum_{i=3}^{M}\frac{(i-1)(i-2)}{2}|X_i|\\
=&2|X_1|+3|X_2|+\sum_{i=3}^{M}\frac{5i-i^2}{2}|X_i|\\
\le&2|X_1|+3|X_2|+3|X_3|+2|X_4|\\
\le&2(2n^2+1-a_1-a_2)+3(a_1+a_2)\\
=&4n^2+2+a_1+a_2.
\end{align*}
This leads to $a_1+a_2\ge 4n-\beta-2\ge 3n-2$, which is a contradiction.
\end{proof}

\subsection{$n\equiv2\pmod{3}$}
\begin{proposition}\label{prop2}
  Theorem~\ref{mainthm} holds for $n\equiv2\pmod{3}$ except for $n=5,8,11$.
\end{proposition}
\begin{proof}
By Equation (\ref{eq4}), we have
\begin{align}\label{eq12}
TT^{(2)}\equiv 2H+2T+2T^{(3)}\pmod{3}.
\end{align}
From above equation, for any $h\in H\backslash(T\cup T^{(3)})$, $h$ appears at least twice in $TT^{(2)}$.
Note that for any $t\in T$, $t=t^{-1}t^2$ and $t^3=tt^2$. Hence, $|X_0|=0$. By Lemma~\ref{lemma3} and Equation~(\ref{eq12}), we get
\begin{align*}
&\sum_{i\ge1}|X_{3i}|=a_1,\\
&\sum_{i\ge0}|X_{3i+1}|=2n+a_2,\\
&\sum_{i\ge0}|X_{3i+2}|=2n^2-2n+1-a_1-a_2,
\end{align*}
where
\begin{align*}
&a_1=|\{h\in T^{(3)}:\ h\text{ appears }3i+2 \text{ times in }T^{(3)}\text{ for some }i\}|,\\
&a_2=|\{h\in T^{(3)}:\ h\text{ appears }3i+1 \text{ times in }T^{(3)}\text{ for some }i\}|.
\end{align*}
Then $a_1+a_2\le 2n$. By Lemma~\ref{lemma2}, we obtain
\begin{equation}\label{eq5}
\begin{split}
\sum_{i=3}^{M}(i-2)|X_i|=&\sum_{i=1}^{M}i|X_i|-2\sum_{i=1}^{M}|X_i|+|X_1|\\
=&4n^2-2(2n^2+1)+|X_1|\\
=&|X_1|-2,
\end{split}
\end{equation}
and
\begin{equation}\label{eq7}
\begin{split}
2n^2+1-4n+\beta=&\sum_{i=3}^{M}\frac{(i-1)(i-2)}{2}|X_i|\\
\le&\frac{M-1}{2}\sum_{i=3}^{M}(i-2)|X_i|\\
=&\frac{M-1}{2}(|X_1|-2).
\end{split}
\end{equation}
 If $s,t\in T$ such that $s=t^2\in T\cap T^{(2)}$, then $s^2=t^4\in T^{(2)}$. Hence $t^3=t\cdot t^2=t^{-1}\cdot s^2\in TT^{(2)}$, and so $t^3\notin X_1$. Thus $|X_1|\le4n-2\beta$. Hence,
\begin{align}\label{eq6}
M\ge\frac{4n^2-8n+2\beta+2}{|X_1|-2}+1\ge\frac{4n^2-8n+2\beta+2}{4n-2\beta-2}+1=\frac{4n^2-4n}{4n-2\beta-2}.
\end{align}
Since $M\in\mathbb{Z}$, then $M\ge n$.

{\bf{Case 1: $|X_{M}|\ge5$.}}

From Equation (\ref{eq5}), we derive that $5n-10\le 5(M-2)\le |X_1|-2\le 4n-2$. This is possible only for $n=5,8$.

{\bf{Case 2: $|X_M|=4.$}}

From Equation (\ref{eq5}), we derive that $4(M-2)\le |X_1|-2\le 4n-2$. Then $M\le n+1$. So $n\le M\le n+1$.

If $M=n+1$, then
\begin{align*}
2n^2-4n+\beta+1=&\sum_{i=3}^{n}\frac{(i-1)(i-2)}{2}|X_i|+4\frac{n(n-1)}{2}\\
=&\sum_{i=3}^{n}\frac{(i-1)(i-2)}{2}|X_i|+2n^2-2n.
\end{align*}
This leads to $\beta\ge2n-1$, which is a contradiction.

If $M=n$, then we have
\[\sum_{i=3}^{n-1}\frac{(i-1)(i-2)}{2}|X_i|=(2n^2-4n+\beta+1)-4\frac{(n-1)(n-2)}{2}=2n+\beta-3,\]
and
\[\sum_{i=3}^{n-1}(i-2)|X_i|=(|X_1|-2)-4(n-2)=|X_1|-4n+6\le6.\]
Let $L$ be the maximum of $3\le i\le n-1$ such that $|X_i|\ne0$, then $L\le8$. This implies that
\[2n+\beta-3=\sum_{i=3}^{n-1}\frac{(i-1)(i-2)}{2}|X_i|\le\frac{L-1}{2}\sum_{i=3}^{n-1}(i-2)|X_i|\le\frac{8-1}{2}6=21.\]
Thus $21\ge 2n+\beta-3\ge 3n-3$. This is possible only for $n=5,8$.

{\bf{Case 3: $|X_M|=3.$}}

Since $X_{M}=X_{M}^{(-1)}$, then $e\in X_{M}$. Note that $e$ appears $|T\cap T^{(2)}|=2\beta$ times, then $M=2\beta$, and $|X_1|\le 4n-M$. By Equation~(\ref{eq6}), we have
\begin{align*}
M\ge\frac{4n^2-8n+2\beta+2}{|X_1|-2}+1\ge\frac{4n^2-4n}{4n-M-2}.
\end{align*}
This leads to $M^2-(4n-2)M+4n^2-4n\le0$, so $2n-2\le M\le2n$, which contradicts to $3(M-2)\le |X_1|-2$.

{\bf{Case 4: $|X_M|=2.$}}

If $X_M\cap T^{(3)}=\emptyset$, suppose that $t=u_1v_{1}^{2}=\cdots=u_Mv_{M}^{2}\in X_M$, where $u_j,v_j\in T$.
By Lemma~\ref{lemma1}, $u_j,v_j$ ($j\in[1,M]$) are pairwise distinct. Then $M\le n$. Combining with Equation (\ref{eq6}), we have $M=n$ and $\beta=1$, $|X_1|=4n-2$. We obtain
\begin{align*}
&\sum_{i=3}^{n-1}|X_i|(i-2)=2n,\\
&\sum_{i=3}^{n-1}|X_i|\frac{(i-1)(i-2)}{2}=n^2-n.
\end{align*}
This implies that
\begin{align*}
n^2-n\le\frac{n-2}{2}\sum_{i=3}^{n-1}|X_i|(i-2)\le(n-2)n,
\end{align*}
which is a contradiction. Therefore $X_M\subset T^{(3)}$.

Suppose that $t_{1}^{3}=\cdots=t_{s}^{3}=u_1v_{1}^{2}=\cdots=u_rv_{r}^{2}\in X_M$, where $t_i,u_j,v_j\in T$, $s\ge1$ and $s+r=M$. By Lemma~\ref{lemma1},  $t_i,u_j,v_j$ ($i\in[1,s], j\in[1,r]$) are pairwise distinct. Then we have $s+2r\le2n$. Since $t_{i}^{3},t_{i}^{-3}\notin X_{1}$ for $i\in[1,s]$, then $|X_1|\le4n-2s$. By Equation~(\ref{eq6}), we have
\begin{align*}
M\ge\frac{4n^2-8n+2\beta+2}{|X_1|-2}+1\ge\frac{4n^2-8n+4}{|X_1|-2}+1.
\end{align*}
If $|X_1|\le4n-2s-2$, then we can get $M\ge n+\frac{s+1}{2}$. This leads to $s+2r\ge2n+1$, which is a contradiction. Note that $e\notin X_1$ and $X_{1}^{(-1)}=X_1$, then $|X_1|$ is even. So $|X_1|=4n-2s$ and $T^{(3)}\subset(X_1\cup X_M)$. By Equation~(\ref{eq6}) again, we obtain
\begin{align*}
M\ge\frac{4n^2-8n+4}{4n-2s-2}+1=n+\frac{s-1}{2}+\frac{s^2-2s+1}{4n-2s-2}.
\end{align*}
This leads to $M\ge n+\frac{s}{2}.$
Recall that $M=r+s$, $s+2r\le2n$, we obtain $s$ is even and $M=n+\frac{s}{2}$. By Equation~(\ref{eq5}), we can get that
\[2(M-2)\le |X_1|-2,\]
which leads to $s\le\frac{2n+2}{3}$. If $s=\frac{2n+2}{3}$, by Equation~(\ref{eq5}), we have $|X_i|=0$ for all $3\le i\le M$. By Lemma~\ref{lemma2}, we get
\[2n^2+1-4n+\beta=(M-1)(M-2)=(n+\frac{n+1}{3}-1)(n+\frac{n+1}{3}-2),\]
which is impossible. Hence $s<\frac{2n+2}{3}$, and there exists $i\in[3,M-1]$ such that $|X_i|\ne0$. Let $L$ be the maximum of $i\in[3,M-1]$ with $|X_i|\ne0$, then $X_L\cap T^{(3)}=\emptyset$ and $L\le n$. By Equations (\ref{eq5}) and (\ref{eq7}), we have
\begin{align}\label{eq16}
\sum_{i=3}^{L}(i-2)|X_i|=(|X_1|-2)-2(M-2)=4n-2s-2-2n-s+4=2n-3s+2,
\end{align}
and
\begin{equation}\label{eq17}
\begin{split}
\sum_{i=3}^{L}\frac{(i-1)(i-2)}{2}|X_i|=&(2n^2+1-4n+\beta)-2\frac{(M-1)(M-2)}{2}\\
=&n^2-(1+s)n-\frac{s^2}{4}+\frac{3s}{2}+\beta-1\\
\ge&n^2-(1+s)n-\frac{s^2}{4}+\frac{3s}{2}.
\end{split}
\end{equation}
This implies that
\begin{align*}
n^2-(1+s)n-\frac{s^2}{4}+\frac{3s}{2}\le&\sum_{i=3}^{L}\frac{(i-1)(i-2)}{2}|X_i|\\
\le&\frac{L-1}{2}\sum_{i=3}^{L}(i-2)|X_i|\\
\le&\frac{n-1}{2}(2n-3s+2),
\end{align*}
which leads to $s^2-2ns+4n-4\ge0$. Thus $s\le2$, and so $s=2$ and $L=n$. Note that $|X_1|=4n-4\le 4n-2\beta$, then $\beta\le2$. Hence $e\notin X_L$. Since $X_{L}^{(-1)}=X_L$, then $|X_L|\ge2$. By Equation~(\ref{eq16}), we have $|X_L|=2$.

Now we consider the elements in $X_L=X_n$ and $X_M=X_{n+1}$. Suppose that $t_{1}^{3}=t_{2}^{3}=u_1v_{1}^{2}=\cdots=u_{n-1}v_{n-1}^{2}\in X_{n+1}$ and $a_1b_{1}^{2}=\cdots=a_nb_{n}^{2}\in X_n$,
where $t_i,u_j,v_j\in T$ ($i\in[1,2],j\in[1,n-1]$) are pairwise distinct and $a_i,b_i\in T$ ($i\in[1,n]$) are pairwise distinct.

If $\{b_i: i\in[1,n]\}\cap\{b_{i}^{-1}: i\in[1,n]\}\ge2$, without loss of generality, we may assume $b_i=b_{i+2}^{-1}$ for $i=1,2$. From $a_1b_{1}^{2}=a_2b_{2}^{2}$ and $a_{3}^{-1}b_{3}^{-2}=a_{4}^{-1}b_{4}^{-2}$, we have
\begin{align*}
b_{1}^{2}b_{2}^{-2}=b_{3}^{-2}b_{4}^{2}=a_{1}^{-1}a_2=a_{3}a_{4}^{-1}.
\end{align*}
By Lemma~\ref{lemma1}, we get $a_{1}^{-1}=a_3$, then $a_1b_{1}^{2}=a_{3}^{-1}b_{3}^{-2}$, which contradicts to $|X_L|=2$.
Then we may assume that $\{t_1,t_2,v_i: i\in[1,n-1]\}\cap\{b_i: i\in[1,n]\}\ge3$ (otherwise, we replace $\{b_i: i\in[1,n]\}$ by $\{b_{i}^{-1}: i\in[1,n]\}$).
Without loss of generality, we may assume $v_i=b_i$ for $i\in[1,3]$. By $u_1v_{1}^{2}=u_2v_{2}^{2}=u_3v_{3}^{2}$ and $a_1b_{1}^{2}=a_2b_{2}^{2}=a_3b_{3}^{2}$, we obtain
\begin{align*}
&v_{1}^{2}v_{2}^{-2}=b_{1}^{2}b_{2}^{-2}=u_{1}^{-1}u_2=a_{1}^{-1}a_2,\\
&v_{1}^{2}v_{3}^{-2}=b_{1}^{2}b_{3}^{-2}=u_{1}^{-1}u_3=a_{1}^{-1}a_3.
\end{align*}
By Lemma~\ref{lemma1}, we get $a_{1}^{-1}=u_2=u_3$, which is a contradiction.

{\bf{Case 5: $|X_M|=1.$}}

For this case, we have $X_M=\{e\}$, $M$ is even and $M=2\beta$. With a similar discussion as Case 3, we can get $2n-2\le M\le 2n$.

If $M=2n-2$, then $|X_1|\le4n-2\beta=2n+2$. By Equations (\ref{eq5}) and (\ref{eq7}), we have
\begin{align}
&\sum_{i=3}^{M-1}|X_i|(i-2)\le4,\label{eq18}\\
&\sum_{i=3}^{M-1}|X_i|\frac{(i-1)(i-2)}{2}=2n-6.\label{eq19}
\end{align}
Let $L$ be the maximum $i\in[3,M-1]$ such that $|X_i|\ne0$, then $|X_L|\ge2$. By Equation~(\ref{eq18}), we have $L\le4$. From Equation (\ref{eq19}), we can get $2n-6\le\frac{L-1}{2}4\le6$. This is possible only for $n=5$.

If $M=2n$, then $T=T^{(2)}$. We obtain $T^2=2H-T+(2n-2)e$. This means that $T$ is a partial difference set with $v=2n^2+1$, $k=2n$, $\lambda=1$, $\mu=2$ and $\Delta=8n-7$. Note that $\gcd(2n^2+1,8n-7)=\gcd(81,8n-7)$, by Lemma~\ref{lemma4}, we have $2n^2+1=3^a$ and $8n-7=3^{b}$ for some integers $a,b$. This is possible only for $n=11$.
\end{proof}

\subsection{$n\equiv0\pmod{3}$}
By Equation (\ref{eq4}), we have
\[TT^{(2)}\equiv T+2T^{(3)}\pmod{3}.\]
Since $\gcd(3,2n^2+1)=1$, then all coefficients of elements in $T^{(3)}$ is 1. By Lemma~\ref{lemma3}, we have
\begin{align*}
&\sum_{i\ge0}|X_{3i+1}|=2n,\\
&\sum_{i\ge0}|X_{3i+2}|=2n,\\
&\sum_{i\ge0}|X_{3i}|=2n^2-4n+1.
\end{align*}
By Lemma~\ref{lemma1}, we obtain that $|X_1|=2n$, $X_1=T$ and $|X_{3i+1}|=0$ for all $i\ge1$.
\begin{lemma}\label{lemma7}
$\beta=0$, i.e.,  $T\cap T^{(2)}=\emptyset$, and $|X_i|=0$ for $i\ge5$.
\end{lemma}
\begin{proof}
Note that $e\in\cup_{i\ge0}X_{3i}$,
 if $\beta>0$, then $\beta\ge3$ and $e\in X_{2\beta}$. By Lemma~\ref{lemma2}, we have
  \begin{align*}
  \beta=&(2n^2+1-4n+\beta)-(2n^2-4n+1)\\
  =&(|X_0|+\sum_{i\ge3}|X_i|\frac{(i-1)(i-2)}{2})-(\sum_{i\ge0}|X_{3i}|)\\
  =&\sum_{i\ge1}|X_{3i+2}|\binom{3i+1}{2}+\sum_{i\ge2}|X_{3i}|\frac{3i(3i-3)}{2}\\
  \ge&\frac{2\beta(2\beta-3)}{2}.
  \end{align*}
  This leads to $\beta^2-2\beta\le0$, which is a contradiction. Hence $\beta=0$. Substituting $\beta=0$ into above equation, we get $|X_i|=0$ for $i\ge5$.
\end{proof}
From above lemma, we obtain
\begin{align*}
&|X_0|+|X_3|=2n^2-4n+1,\\
&|X_1|=|X_2|=2n,\\
&|X_i|=0\text{ for all }i\ge4.
\end{align*}
By Lemma~\ref{lemma2}, we get $|X_0|=\frac{2}{3}n^2-2n+1$, $|X_3|=\frac{4}{3}n^2-2n$, and
\begin{align}\label{eq11}
TT^{(2)}=T+2T^{(3)}+3X_3.
\end{align}

In order to get a contradiction, we need to investigate $TT^{(4)}$. By Theorem~\ref{groupring}, we have
\begin{align*}
(T^{(2)})^2=&(2H+(2n-2)e-T^2)^2\\
=&(8n^2+8n-4)H+(4n^2-8n+4)e-(4H+(4n-4)e)T^2+T^4\\
=&(8n^2+8n-4)H+(4n^2-8n+4)e-(4H+(4n-4)e)(2H-T^{(2)}+(2n-2)e)+T^4\\
=&(-8n^2+4)H+(-4n^2+8n-4)e+(4n-4)T^{(2)}+T^4.
\end{align*}
On the other hand, $T^{(2)}$ also satisfies conditions in Theorem~\ref{groupring}, then
\[(T^{(2)})^2=2H-T^{(4)}+(2n-2)e.\]
Combining above two equations, we obtain
 \[T^4=(8n^2-2)H+(4n^2-6n+2)e-T^{(4)}-(4n-4)T^{(2)}.\]
  Multiplying $T$ to both sides, we can get
\begin{align}\label{eq8}
TT^{(4)}=(16n^3-4n)H+(4n^2-6n+2)T-(4n-4)TT^{(2)}-T^{5}.
\end{align}
Writing $TT^{(4)}$ as
\begin{align*}
TT^{(4)}=\sum_{i=0}^{N}iY_i,
\end{align*}
where $N\le2n$, and $Y_i$, $i=0,1,\dots,N$ form a partition of group $H$, i.e.,
\begin{align*}
H=\sum_{i=0}^{N}Y_i.
\end{align*}
Then we have the following lemma.
\begin{lemma}\label{lemma5}
  \begin{enumerate}
    \item [(1)] $\sum_{i=0}^{N}|Y_i|=2n^2+1$,
    \item [(2)] $\sum_{i=1}^{N}i|Y_i|=4n^2,$
    \item [(3)] $\sum_{i=1}^{N}|Y_i|=4n-\gamma+\sum_{i\ge3}|Y_i|\frac{(i-1)(i-2)}{2}$, where $|T\cap T^{(4)}|=2\gamma$.
  \end{enumerate}
\end{lemma}
\begin{proof}
  The proof is similar to that for Lemma~\ref{lemma2}.
\end{proof}
\begin{corollary}\label{coro1}
  $2|Y_1|+3|Y_2|+3|Y_3|+2|Y_4|\ge4n^2+3n.$
\end{corollary}
\begin{proof}
  By Lemma~\ref{lemma5}, we have
  \begin{align*}
  4n^2+4n-\gamma=&\sum_{i=1}^{N}i|Y_i|+\sum_{i=1}^{N}|Y_i|-\sum_{i\ge3}|Y_i|\frac{(i-1)(i-2)}{2}\\
  =&2|Y_1|+3|Y_2|+3|Y_3|+2|Y_4|-\sum_{i\ge5}|Y_i|\frac{i(i-5)}{2}\\
  \le&2|Y_1|+3|Y_2|+3|Y_3|+2|Y_4|.
  \end{align*}
  Then the statement follows from $\gamma\le n$.
\end{proof}
Now we divide our discussion into five cases according to the residues modulo 5.

\begin{proposition}\label{prop3}
  Theorem~\ref{mainthm} holds for $n\equiv6\pmod{15}$.
\end{proposition}
\begin{proof}
By Equation~(\ref{eq8}), we have
\begin{align}\label{eq10}
TT^{(4)}\equiv2H+4T^{(5)}\pmod{5}.
\end{align}
Since $\gcd(5,2n^2+1)=1$, then all coefficients of $T^{(5)}$ is 1. By Equation (\ref{eq10}), we obtain
\begin{align*}
&\sum_{i\ge0}|Y_{5i+1}|=2n,\\
&\sum_{i\ge0}|Y_{5i+2}|=2n^2-2n+1,\\
&|Y_{i}|=0\text{ for all }i\equiv0,3,4\pmod{5}.
\end{align*}
By Lemma~\ref{lemma5}, we can get
\begin{equation}\label{eq9}
\begin{split}
2n-2=&\sum_{i=1}^{N}i|Y_i|-\sum_{i\ge0}|Y_{5i+1}|-2\sum_{i\ge0}|Y_{5i+2}|\\
=&\sum_{i\ge1}5i(|Y_{5i+1}|+|Y_{5i+2}|).
\end{split}
\end{equation}
Consequently,
\begin{align*}
2n^2+1-4n+\gamma=&\sum_{i=6}^{N}|Y_i|\frac{(i-1)(i-2)}{2}\\
\le&\frac{N-1}{2}\sum_{i=6}^{N}|Y_i|(i-2)\\
\le&\frac{N-1}{2}(2n-2).
\end{align*}
This implies that $N\ge2n-1$. From Equation~(\ref{eq9}), we have $|Y_N|=1$, and so $Y_{N}=\{e\}$ and $N=|T\cap T^{(4)}|$ is even. Thus $N=2n$.
In this case, we have $T=T^{(4)}$ and $TT^{(4)}=T^2=2H-T^{(2)}+(2n-2)e$. Comparing with Equation~(\ref{eq10}), we obtain $T^{(2)}=T^{(5)}$. Then for any $t\in T$, there exists $s\in T$ such that $s^2=t^5=t\cdot t^4$. Note that $t^4\in T$, by Lemma~\ref{lemma1}, we have $t=t^4=s$. Then $t^3=e$, and so $t=e$, which is a contradiction.
\end{proof}

\begin{proposition}\label{prop4}
  Theorem~\ref{mainthm} holds for $n\equiv9\pmod{15}$.
\end{proposition}
\begin{proof}
By Equations~(\ref{eq11}) and (\ref{eq8}), we have
\begin{align*}
TT^{(4)}\equiv&3H+2T+3TT^{(2)}+4T^{(5)}\pmod{5}\\
\equiv&3H+2T+3T+T^{(3)}+4X_3+4T^{(5)}\pmod{5}\\
\equiv&2T+4T^{(5)}+3X_0+T+4T^{(3)}+2X_3\pmod{5}\\
\equiv&4T^{(5)}+3X_0+3T+4T^{(3)}+2X_3\pmod{5}.
\end{align*}
From above equation, we can get
\begin{align*}
&\sum_{i\ge0}|Y_{5i+1}|=|T^{(5)}\cap X_3|,\\
&\sum_{i\ge0}|Y_{5i+2}|=|X_3\backslash T^{(5)}|+|T^{(5)}\cap X_0|+|T^{(5)}\cap T|,\\
&\sum_{i\ge0}|Y_{5i+3}|=|X_0\backslash T^{(5)}|+|T\backslash T^{(5)}|+|T^{(3)}\cap T^{(5)}|,\\
&\sum_{i\ge0}|Y_{5i+4}|=|T^{(3)}\backslash T^{(5)}|.
\end{align*}
By Lemma~\ref{lemma5}, we obtain
\begin{align*}
4n^2=&\sum_{i=1}^{N}i|Y_i|\\
\ge&\sum_{i\ge0}|Y_{5i+1}|+2\sum_{i\ge0}|Y_{5i+2}|+3\sum_{i\ge0}|Y_{5i+3}|+4\sum_{i\ge0}|Y_{5i+4}|\\
=&|T^{(5)}\cap X_3|+2(|X_3\backslash T^{(5)}|+|T^{(5)}\cap X_0|+|T^{(5)}\cap T|)+3(|X_0\backslash T^{(5)}|+|T\backslash T^{(5)}|+|T^{(3)}\cap T^{(5)}|)\\
&+4|T^{(3)}\backslash T^{(5)}|\\
\ge&3(\frac{2}{3}n^2-2n+1)+3\times2n+4\times2n+2(\frac{4}{3}n^2-2n)-2n\\
=&\frac{14}{3}n^2+2n+3,
\end{align*}
which is impossible.
\end{proof}

\begin{proposition}\label{prop5}
  Theorem~\ref{mainthm} holds for $n\equiv3\pmod{15}$ except for $n=3$.
\end{proposition}
\begin{proof}
By Equations~(\ref{eq11}) and (\ref{eq8}), we have
\begin{align*}
TT^{(4)}\equiv&2TT^{(2)}+4T^{(5)}\pmod{5}\\
\equiv&2T+4T^{(3)}+X_3+4T^{(5)}\pmod{5}.
\end{align*}
From above equation, we can get
\begin{align*}
&\sum_{i\ge0}|Y_{5i+1}|=|X_3\backslash T^{(5)}|+|T\cap T^{(5)}|,\\
&\sum_{i\ge0}|Y_{5i+2}|=|T\backslash T^{(5)}|,\\
&\sum_{i\ge0}|Y_{5i+3}|=|T^{(3)}\cap T^{(5)}|,\\
&\sum_{i\ge0}|Y_{5i+4}|=|T^{(3)}\backslash T^{(5)}|+|T^{(5)}\cap X_0|.
\end{align*}
By Corollary~\ref{coro1}, we obtain
\begin{align*}
4n^2+3n\le&2|Y_1|+3|Y_2|+3|Y_3|+2|Y_4|\\
\le& 2(|X_3\backslash T^{(5)}|+|T\cap T^{(5)}|+|T^{(3)}\backslash T^{(5)}|+|T^{(5)}\cap X_0|)+3(|T\backslash T^{(5)}|+|T^{(3)}\cap T^{(5)}|)\\
\le& 2(\frac{4}{3}n^2-2n+2n+2n)+3\times 2n.
\end{align*}
This leads to $\frac{4}{3}n^2-7n\le0$ and so $n\le5$. Hence $n=3$.
\end{proof}

\begin{proposition}\label{prop6}
  Theorem~\ref{mainthm} holds for $n\equiv12\pmod{15}$.
\end{proposition}
\begin{proof}
By Equations~(\ref{eq11}) and (\ref{eq8}), we have
\begin{equation}\label{eq13}
\begin{split}
TT^{(4)}\equiv& T+TT^{(2)}+4T^{(5)}\pmod{5}\\
\equiv&2T+2T^{(3)}+3X_{3}+4T^{(5)}\pmod{5}.
\end{split}
\end{equation}
From above equation, we can get
\begin{align*}
&\sum_{i\ge0}|Y_{5i+1}|=|T\cap T^{(5)}|+|T^{(3)}\cap T^{(5)}|,\\
&\sum_{i\ge0}|Y_{5i+2}|=|T\backslash T^{(5)}|+|T^{3}\backslash T^{(5)}|+|X_3\cap T^{(5)}|,\\
&\sum_{i\ge0}|Y_{5i+3}|=|X_3\backslash T^{(5)}|,\\
&\sum_{i\ge0}|Y_{5i+4}|=|X_0\cap T^{(5)}|.
\end{align*}
By Lemma~\ref{lemma5}, we obtain
\begin{align*}
4n^2=&\sum_{i=1}^{N}i|Y_i|\\
\ge&\sum_{i\ge0}|Y_{5i+1}|+2\sum_{i\ge0}|Y_{5i+2}|+3\sum_{i\ge0}|Y_{5i+3}|+4\sum_{i\ge0}|Y_{5i+4}|\\
=&(|T\cap T^{(5)}|+|T^{(3)}\cap T^{(5)}|)+2(|T\backslash T^{(5)}|+|T^{3}\backslash T^{(5)}|+|X_3\cap T^{(5)}|)+3|X_3\backslash T^{(5)}|+4|X_0\cap T^{(5)}|\\
\ge&3(\frac{4}{3}n^2-2n)+2\times2n+2\times2n-2n\\
=&4n^2.
\end{align*}
This means that $T^{(5)}\subset(T\cup T^{(3)}\cup X_3)$ and
\begin{align*}
&|Y_1|=a_1,\\
&|Y_2|=4n-a_1+a_2,\\
&|Y_3|=\frac{4}{3}n^2-2n-a_2,\\
&|Y_i|=0\text{ for all }i\ge4,
\end{align*}
where $a_1=|T^{(5)}\cap (T\cup T^{(3)})|$, $a_2=|T^{(5)}\cap X_3|$ and $a_1+a_2=2n$.
Since $e\in X_0$, then $T\cap T^{(4)}=\emptyset$. By Lemma~\ref{lemma5}, we have $|Y_1|+|Y_2|=4n$. So $a_2=0$ and $a_1=2n$. Thus $T^{(5)}\subset(T\cup T^{(3)})$.

If there exist $s,t\in T$ such that $s^3=t^5\in (T^{(5)}\cap T^{(3)})$, then $s^3=t^5=s^{-1}\cdot s^4=t\cdot t^4$. So $s^3=t^5$ appears at least twice in $TT^{(4)}$, which is a contradiction. Hence $T=T^{(5)}$ and
\[TT^{(4)}=T+2T^{(3)}+3X_{3}.\]
 Comparing with Equation~(\ref{eq11}), we have $TT^{(2)}=TT^{(4)}$, i.e., $T(T^{(2)}-T^{(4)})=0$.
Then for any nontrivial character $\chi\in\widehat{H}$, we have $\chi(T)(\chi(T^{(2)})-\chi(T^{(4)}))=0$.
 By Lemma~\ref{lemma6}, $\chi(T)\ne0$, then $\chi(T^{(2)})-\chi(T^{(4)})=0$ for any $\chi\in\widehat{H}$. Hence $T^{(2)}=T^{(4)}$, which contradicts to Lemma~\ref{lemma7}.
\end{proof}

\begin{proposition}\label{prop7}
  Theorem~\ref{mainthm} holds for $n\equiv0\pmod{15}$.
\end{proposition}
\begin{proof}
By Equations~(\ref{eq11}) and (\ref{eq8}), we have
\begin{align*}
T^6=&(32n^4-8n^2)H+(4n^2-6n+2)T^2-(4n-4)T^2T^{(2)}-T^2T^{(4)}\\
=&(32n^4-8n^2)H+(4n^2-6n+2)(2H-T^{(2)}+(2n-2)e)-(4n-4)(2H-T^{(2)}+(2n-2)e)T^{(2)}\\
&-(2H-T^{(2)}+(2n-2)e)T^{(4)}\\
=&(32n^4-16n^2+4)H-(12n^2-22n+10)T^{(2)}+(4n-4)(T^{(2)})^2+T^{(2)}T^{(4)}-(2n-2)T^{(4)}+\\
&(8n^3-20n^2+16n-4)e\\
=&(32n^4-16n^2+4)H-(12n^2-22n+10)T^{(2)}+(4n-4)(2H-T^{(4)}+(2n-2)e)+\\
&T^{(2)}T^{(4)}-(2n-2)T^{(4)}+(8n^3-20n^2+16n-4)e\\
=&(32n^4-16n^2+8n-4)H-(12n^2-22n+10)T^{(2)}+T^{(2)}T^{(4)}-(6n-6)T^{(4)}\\
&+(8n^3-12n^2+4)e.
\end{align*}
This implies that
\begin{equation}\label{eq14}
\begin{split}
TT^{(5)}\equiv& H+T^{(2)}T^{(4)}+T^{(4)}+4e\pmod{5}\\
\equiv&H+T^{(2)}+2T^{(6)}+3X_{3}^{(2)}+T^{(4)}+4e\pmod{5}\\
\equiv&X_{0}^{(2)}+2T^{(2)}+3T^{(6)}+4X_{3}^{(2)}+T^{(4)}+4e\pmod{5}\\
\equiv&(X_{0}^{(2)}\backslash\{e\})+2T^{(2)}+3T^{(6)}+4X_{3}^{(2)}+T^{(4)}\pmod{5}.
\end{split}
\end{equation}
Writing $TT^{(5)}=\sum_{i=0}^{L}iZ_i$, where $Z_i$, $i=0,1,\dots,L$ form a partition of group $H$. Then
\[4n^2=\sum_{i=1}^{L}i|Z_i|.\]
From Equation (\ref{eq14}), we can get
\begin{align*}
&\sum_{i\ge0}|Z_{5i+1}|=|X_{0}^{(2)}\backslash(\{e\}\cup T^{(4)})|,\\
&\sum_{i\ge0}|Z_{5i+2}|=|T^{(2)}|+|X_{0}^{(2)}\cap T^{(4)}|,\\
&\sum_{i\ge0}|Z_{5i+3}|=|T^{(6)}\backslash T^{(4)}|,\\
&\sum_{i\ge0}|Z_{5i+4}|=|X_{3}^{(2)}\backslash T^{(4)}|+|T^{(6)}\cap T^{(4)}|.
\end{align*}
Hence we obtain
\begin{align*}
4n^2=&\sum_{i=1}^{L}i|Z_i|\\
\ge&\sum_{i\ge0}|Z_{5i+1}|+2\sum_{i\ge0}|Z_{5i+2}|+3\sum_{i\ge0}|Z_{5i+3}|+4\sum_{i\ge0}|Z_{5i+4}|\\
=&|X_{0}^{(2)}\backslash(\{e\}\cup T^{(4)})|+2(|T^{(2)}|+|X_{0}^{(2)}\cap T^{(4)}|)+3|T^{(6)}\backslash T^{(4)}|+4(|X_{3}^{(2)}\backslash T^{(4)}|+|T^{(6)}\cap T^{(4)}|)\\
\ge&\frac{2}{3}n^2-2n+4n+6n+4(\frac{4}{3}n^2-2n)-8n\\
=&6n^2-8n.
\end{align*}
This leads to $2n^2-8n\le0$, and so $n\le4$, which is impossible.
\end{proof}
\subsection{Some sporadic cases}
For $n=3$, it has been pointed out in \cite{E11} that, there exists a linear $DPL(3,6)$ code. For $n=11$, we have the following result.
\begin{proposition}\label{prop8}
  There exists a linear $DPL(11,6)$ code.
\end{proposition}
\begin{proof}
  Let $H=(\mathbb{F}_{3^5},+)$, and $g$ be a primitive element of finite field $\mathbb{F}_{3^5}$. Let $T=\{g^{11i}: i\in[0,21]\}\subset H$, then $T=T^{(-1)}$, $T=T^{(2)}$ and $T^2=2H-T+(2n-2)e$. By Theorem~\ref{groupring}, there exists a linear $DPL(11,6)$ code.
\end{proof}
For $n=5$ or 8, the nonexistence of $DPL(n,6)$ code follows from following proposition.
\begin{proposition}\label{prop9}
If $n\equiv5,8\pmod{9}$ and $4n^2+2$ is a squarefree integer, then there does not exist a linear $DPL(n,6)$ code.
\end{proposition}
\begin{proof}
Let $e_i$, $i=1,2,\dots,n$, be fixed orthonormal basis of $\mathbb{Z}^{n}$.
  By Theorem~\ref{tilinggroup}, there exists a linear $DPL(n,6)$ code if and only if there is a homomorphism $\phi:\mathbb{Z}^n\rightarrow \mathbb{Z}_{4n^2+2}$ (the group $\mathbb{Z}_{4n^2+2}$ is written additively) such that the restriction of $\phi$ to $DS_{n,2}(0,e_1)$ is a bijection. Since the homomorphism $\phi$ is determined by the values $\phi(e_i)$, $i=1,\dots,n$, then there exists a linear $DPL(n,6)$ code if and only if there exists an $n$-subset $\{a_1,a_2,\dots,a_n\}\subset \mathbb{Z}_{4n^2+2}$ (let $\phi(e_i)=a_i$) such that
  \begin{align*}
  \mathbb{Z}_{4n^2+2}=&\{0,\pm a_{1},\pm2a_{1},3a_{1}\}\cup\{\pm a_{i},\pm a_{1}\pm a_{i},2a_{1}\pm a_{i},\pm2a_{i},a_1\pm2a_{i}:\ 2\le i\le n\}\\
  &\cup\{\pm a_{i}\pm a_{j},a_{1}\pm a_{i}\pm a_{j}:\ 2\le i<j\le n\}.
  \end{align*}
  Then we have
  \begin{equation}\label{eq15}
  \begin{split}
  &(\pm a_{1})^2+(\pm2a_{1})^2+(3a_{1})^2+\sum_{i=2}^{n}((\pm a_{i})^2+(\pm a_{1}\pm a_{i})^2+(2a_{1}\pm a_{i})^2+(\pm2a_{i})^2\\
  &+(a_1\pm2a_{i})^2)+\sum_{2\le i<j\le n}((\pm a_{i}\pm a_{j})^2+(a_{1}\pm a_{i}\pm a_{j})^2)\equiv\sum_{i=1}^{4n^2+1}i^2\pmod{4n^2+2}.
  \end{split}
  \end{equation}
  The left hand side of Equation~(\ref{eq15}) can be computed as
  \begin{align*}
  &(\pm a_{1})^2+(\pm2a_{1})^2+(3a_{1})^2+\sum_{i=2}^{n}((\pm a_{i})^2+(\pm a_{1}\pm a_{i})^2+(2a_{1}\pm a_{i})^2+(\pm2a_{i})^2+(a_1\pm2a_{i})^2)\\
  &+\sum_{2\le i<j\le n}((\pm a_{i}\pm a_{j})^2+(a_{1}\pm a_{i}\pm a_{j})^2)\\
  =&(2n^2+8n+9)a_{1}^{2}+(8n+8)\sum_{i=2}^{n}a_{i}^{2}\\
  \equiv&0\pmod{3}.
  \end{align*}
  For the right hand side of Equation~(\ref{eq15}), we have
   \[\sum_{i=1}^{4n^2+1}i^2=\frac{(4n^2+1)(4n^2+2)(8n^2+3)}{6}\not\equiv0\pmod{3}.\]
  This contradicts to $4n^2+2\equiv0\pmod{3}$.
\end{proof}

Theorem~\ref{mainthm} is a combination of Propositions \ref{prop1}, \ref{prop2}, and \ref{prop3}---\ref{prop9}.

\section*{Acknowledgements}
Gennian Ge was partially supported by the National Key Research and Development Program of China under Grant 2020YFA0712100 and Grant 2018YFA0704703, the National Natural Science Foundation of China under Grant 11971325 and Grant 12231014, and Beijing Scholars Program.

\end{document}